\documentclass[11pt,reqno]{amsart}   	
\usepackage{geometry}                		
\usepackage[utf8]{inputenc}
\usepackage{graphicx}
\usepackage[english]{babel}
\usepackage{amsmath}
\usepackage{amsfonts}
\usepackage{amsthm}
\usepackage{color}
\usepackage{float}
\usepackage{caption}
\usepackage{subcaption}
\usepackage[noend,algonl]{algorithm2e}
\SetKwProg{Fn}{Function}{}{end}

\newtheorem{claim}{Lemma}
\newtheorem{proposition}{Proposition}
\newtheorem{theorem}{Theorem}
\newtheorem*{corollary}{Corollary}
\newtheorem*{remark}{Remark}
\newtheorem{definition}{Definition}

\newcommand{\K}{\mathcal{K}}
\graphicspath{{./}}

\title{Born and inverse Born series for scattering problems with Kerr nonlinearities}

\author{Nicholas DeFilippis}
\address{Department of Mathematics, Drexel University, Philadelphia, PA, USA}
\email{}

\author{Shari Moskow}
\address{Department of Mathematics, Drexel University, Philadelphia, PA, USA }
\email{moskow@math.drexel.edu}

\author{John C. Schotland }
\address{Department of Mathematics and Department of Physics, Yale University, New Haven, CT, USA}
\email{john.schotland@yale.edu}

\date{\today}

\begin{document}

\maketitle
\begin{abstract} 
We consider the Born and inverse Born series for scalar waves with a cubic nonlinearity of Kerr type. We find a recursive formula for the operators in the Born series and prove their boundedness. This result gives conditions which guarantee convergence of the Born series, and subsequently yields conditions which guarantee convergence of the inverse Born series. We also use fixed point theory to give alternate explicit conditions for convergence of the Born series. We illustrate our results with numerical experiments.
\end{abstract}

\section{Introduction}
There has been considerable recent interest in inverse scattering problems for nonlinear partial differential equations (PDEs)~\cite{assylbekov_1,assylbekov_2,carstea,imanuvilov,isakov_1,isakov_2,isakov_3,isakov_4,kang,kurylev,lassas}. There 
are numerous applications in various applied fields ranging from optical imaging to seismology. In general terms, the problem to be considered is to reconstruct the coefficients of a nonlinear PDE from boundary measurements. As in the case of inverse problems for linear PDEs, the fundamental questions relate to the 
uniqueness, stability and reconstruction of the unknown coefficients. In contrast to the linear case (which is very well studied), the study of inverse problems for nonlinear PDE is still relatively unexplored.  There are uniqueness and stability results for a variety of semilinear and quasilinear equations. Reconstruction methods for such problems are just beginning to be developed\cite{carstea,griesmaier,kang}.

In this paper, we consider the inverse problem of recovering the coefficients of a nonlinear elliptic PDE with cubic nonlinearity. This problem appears in optical physics, where the cubic term arises in the study of the Kerr effect---a nonlinear process where self-focusing of light is observed~\cite{boyd}. We show that it is possible to reconstruct the coefficients of the linear and nonlinear terms in a Kerr medium from boundary measurements. This result holds under a smallness condition on the measurements, which also guarantees the stability of the recovery. The reconstruction is based on inversion of the Born series, which expresses the solution to the inverse problem as an explicitly computable functional of the measured data. We note that this method has been extensively studied  for linear PDEs~\cite{review}. The extension to the nonlinear setting involves a substantial reworking of the theory, especially the combinatorial structure of the Born series itself. We validate our results with numerical simulations that demonstrate the convergence of the inverse Born series under the expected smallness conditions.

The remainder of this paper is organized as follows. In section 2, we introduce the forward problem and state sufficient conditions for its solvability. The Born series is studied in section 3, the combinatorial structure of the series is characterized, and sufficient conditions for convergence are established. We also derive various estimates that are later used in section 4 to obtain our main result on the convergence of the inverse Born series. In section 5, we present the results of numerical reconstructions of a two-dimensional medium. Our conclusions are presented in section 6. The Appendix contains the proof of Proposition~1.

\section{Forward problem}

We consider the Kerr effect in a bounded domain $\Omega$ in $\mathbb{R}^d$ for $d>2$ with a smooth boundary $\partial\Omega$.
The scalar field $u$ obeys the nonlinear PDE
\begin{align}
\label{baseequation}
\Delta u + k^2(1 + \alpha(x))u + k^2\beta(x) |u|^2  u &= 0 \quad \text{ in } \quad \Omega \ , \\
\frac{\partial u}{\partial \nu } &= g \quad  \text{ on } \quad \partial\Omega \ ,
\end{align} 
where $k$ is the wavenumber and $\nu$ is the unit outward normal to $\partial\Omega$. The coefficients $\alpha$ and $\beta$ are the linear and nonlinear susceptibilities, respectively~\cite{boyd} and are taken to be real valued, as is the boundary source $g$. It follows that $u$ is real valued so that $|u|^2u= u^3$. More generally, $u$ is complex valued, in which case our results carry over with small modifications.  

We now consider the solution $u_0$ to the linear problem 
\begin{align}
\label{backgroundequation}
    \Delta u_0 + k^2u_0 &= 0 \quad \text{ in } \quad \Omega \ , \\
    \frac{\partial u_0}{\partial \nu } &= g   \quad  \text{ on } \quad \partial\Omega \ .
    \end{align}
Following standard procedures, we find that the field $u$ obeys the integral equation
\begin{equation}\label{integralequation}
    u(x) = u_0(x) - k^2\int_\Omega G(x, y) \left(\alpha(y)u(y) + \beta(y)u^3(y)\right)dy \ .
\end{equation}
Here the Green's function $G$ obeys
\begin{align}
	\Delta_x G(x, y) + k^2 G(x, y)  &= \delta(x-y)  \quad \text{ in } \quad \Omega \ , \\
	 \frac{\partial G }{\partial \nu_y} &= 0   \quad  \text{ on } \quad \partial\Omega \ .
\end{align}
 
We define the nonlinear operator $T: C(\overline{\Omega})\rightarrow C(\overline{\Omega})$ by
\begin{equation}\label{Tdef}
    T(u) = u_0 - k^2\int_\Omega G(x, y) \left(\alpha(y)u(y) + \beta(y) u^3(y) \right) dy.
\end{equation}
We note that if $u\in C(\overline{\Omega})$ is a fixed point of $T$, that is $u = T(u)$, then $u$ satisfies equation (\ref{integralequation}). 

The following result provides conditions for existence of a unique solution to (\ref{integralequation}). 
\begin{proposition}
\label{bounds}  
Let $T: C(\overline{\Omega})\rightarrow C(\overline{\Omega})$ be defined by (\ref{Tdef}) and define $\mu$ by 
\begin{equation}\label{mudef1} 
 \mu = k^2\sup_{x\in \Omega} \int_\Omega | G(x,y) | dy. 
 \end{equation}
If there exists  $\gamma > 1/2 $ such that $$\Vert\alpha\Vert_\infty < \frac{2\gamma-1}{2\mu (1+\gamma)}$$ and $$\Vert\beta\Vert_\infty < \frac{1}{2\mu \Vert u_0\Vert^2(1+\gamma)^3},$$ then $T$ has a unique fixed point on the ball  of radius $\gamma\| u_0\|_\infty$ about $u_0$ in $C(\overline{\Omega})$.
\end{proposition}

The proof is presented in Appendix \ref{fixedpointappendix}.

\section{Born series}
The forward problem is to compute the field $u$ as measured on $\partial\Omega$ when $g$ corresponds to a point source on $\partial\Omega$.
The solution to the forward problem is derived by iteration of the integral equation (\ref{integralequation}). We thus obtain
\begin{equation} \label{born_series}
    \phi = K_1 (\zeta) + K_2 (\zeta, \zeta) + K_3( \zeta , \zeta,  \zeta )+\cdots \ ,
\end{equation} 
where $\phi=u-u_0$ and $\zeta :=  (\alpha, \beta)$. The forward operators $$K_n: [L^\infty(\Omega)]^{2n}\rightarrow C(\partial\Omega\times\partial\Omega)$$ are constructed below. We will refer to (\ref{born_series}) as the the Born series.
We note that Proposition~\ref{bounds} guarantees convergence of the Born series.

The forward operator $K_n$ is an $n$-linear operator (multilinear of order $n$) on $[ L^\infty(\Omega)^2]^n$. 
In the following, we do not denote the dependence of $u_0$ on the source explicitly. The first term in the fixed point iteration is 
\begin{equation}
 u_1(x) : = T(u_0)(x) = u_0(x) +  k^2\int_\Omega G(x, y) \left[\alpha(y) u_0(y)+ \beta(y)  u_0^3(y)\right]dy \end{equation}
and thus $K_1$ is defined by
\begin{equation}     K_1( \zeta) (x) = k^2\int_\Omega G(x, y) \left[\alpha(y) u_0(y)+ \beta(y) u_0^3(y)\right]dy . 
\end{equation}
Next we observe that
\begin{equation}
u_2(x) : = T(u_1)(x) = u_0(x) + k^2\int_\Omega G(x, y) \left[\alpha(y) u_1(y) + \beta  (y) u_1^3(y)\right]dy.
 \end{equation}
Evidently, expansion of $u_1^3$ leads to terms which are multilinear in $\alpha$ and $\beta$. Subsequent iterates become progressively more complicated. To handle this problem, we introduce the  operators:  $a,b: C(\overline{\Omega}) \times [ L^\infty(\Omega)]^2\rightarrow C(\overline{\Omega})$, which are defined by
\begin{equation} \label{neq}  a(v, \zeta) = k^2\int_{\Omega} G(x, y) \alpha (y) v(y) dy , \end{equation}
and
\begin{equation} \label{beq} b(v, \zeta  ) = k^2 \int_{\Omega} G(x, y) \beta (y) v(y) dy . \end{equation} 

The above operators have tensor counterparts which are defined as follows.
\begin{definition} \label{def1}  Given $T_l = T_l (\zeta_1,\cdots,\zeta_l ),$  a multi-linear operator of order $l$,  define the  $l+1$ order multilinear operators  $BT_l$  and $AT_l$ by
$$ BT_l(\zeta_1, \ldots , \zeta_l, \zeta_{l+1} )= b( T_l (\zeta_1, \ldots , \zeta_l ) , \zeta_{l+1} )$$
and
$$ AT_l(\zeta_1, \ldots , \zeta_l, \zeta_{l+1} )= a( T_l (\zeta_1, \ldots , \zeta_l ) , \zeta_{l+1} )$$
where $b$ and $n$ are given by (\ref{beq}) and (\ref{neq}) respectively.
\end{definition}
We will also need a tensor product of multilinear operators.
\begin{definition} \label{def2} 
Given $T_j$ and $T_l$, multilinear operators of order $j$ and $l$ respectively, define the tensor product $T_l\otimes T_j$ by
$$ T_l\otimes T_j (\zeta_1, \ldots , \zeta_l, \zeta_{l+1},\dots, \zeta_{l+j} )=  T_l (\zeta_1, \ldots , \zeta_l ) T_j (\zeta_{l+1},\dots, \zeta_{l+j} ).$$
and note that $T_l\otimes  T_j$ is a multilinear operator of order $l+j$. 
\end{definition}
Note that the tensor product of multilinear operators does not commute. Tensor products are extended to sums of multilinear operators by bilinearity of the tensor product, and the tensor product is also associative.  In this notation, we see that if $v$ is a sum of multilinear operators, then 
$$ Tv = u_0 + A v+ B v\otimes v \otimes v $$
yields another sum of multilinear operators ($u_0$ is an order zero operator).
\begin{claim} Viewing the $n$th iterate $u_n$ as a sum of multilinear operators, for any $n$ we have that 
\begin{equation} u_{n} = u_{n-1} + \mbox{multilinear operators of order}\  \geq n . \end{equation}
\label{termslem}\end{claim}
\begin{proof}
We will prove this by induction. For the base case $n=1$, we have that $u_1 = u_0 + Au_0+Bu_0\otimes u_0 \otimes u_0$, so the statement holds.
Now assume that the statement holds for $u_{n-1}$. Then
\begin{eqnarray} u_n 
&=& u_0 + Au_{n-1} + Bu_{n-1}\otimes u_{n-1} \otimes u_{n-1} . \end{eqnarray}
By inductive hypothesis, 
$$u_{n-1} = u_{n-2}+ w $$
where $w$ is a sum of operators of order at least $n-1$. Hence we have that 
\begin{multline} u_{n-1}\otimes u_{n-1} \otimes u_{n-1} = u_{n-2}\otimes u_{n-2}\otimes u_{n-2} + u_{n-2}\otimes u_{n-2} \otimes w  \\+ w\otimes u_{n-2}\otimes u_{n-2} +  u_{n-2}\otimes w \otimes u_{n-2}  \\  +  u_{n-2} \otimes w\otimes w +  w\otimes u_{n-2} \otimes w +  u_{n-2}\otimes w\otimes w+ w\otimes w\otimes w , \end{multline}
so that $$u_{n-1} \otimes u_{n-1} \otimes u_{n-1}= u_{n-2}\otimes u_{n-2} \otimes u_{n-2} + \mbox{multilinear operators of order} \geq n-1.$$
Applying $A$ to $u_{n-1}$ and $B$ to $u_{n-1}\otimes u_{n-1} \otimes u_{n-1}$, we increase the order of each by one. Hence we have that
\begin{eqnarray} u_n &=& u_0 + Au_{n-2} + Bu_{n-2}\otimes u_{n-2} \otimes u_{n-2}  + \mbox{terms of degree} \geq n \\
&=& u_{n-1}  + \mbox{terms of degree} \geq n . \end{eqnarray}
The result follows from induction.  \end{proof}Given the previous result, we can now define the forward operators. 
\begin{definition} The $n$th term of the forward series, $K_n (\zeta,\ldots, \zeta) $ is defined to be the sum of all multilinear operators of order exactly $n$ in the $nth$ iterate $u_n$. 
\end{definition}

\subsection{General formula for the forward operators}
Using our tensor notation, the forward series is given by iterations of 
$$Tv= u_0 + Av + B v\otimes v\otimes v.$$
Given $u_0$, we have 
\begin{eqnarray}  u_1 &=&  Tu_0 = u_0 + Au_0 + B  u_0  \otimes u_0 \otimes u_0 , \nonumber \\ u_2 &=&  Tu_1 = u_0 + Au_1 + B  u_1  \otimes u_1 \otimes u_1,   \nonumber \\  u_{n+1} &=&  Tu_n = u_0 + Au_n + B  u_n  \otimes u_n \otimes u_n.  \nonumber \end{eqnarray}
Define $U_n$ to be the sum of the first $n$ forward operators, that is,
\begin{eqnarray} U_n &=& \sum_{i=0}^n K_i(\zeta_1,\ldots\zeta_i) \nonumber \\
&=& u_0 + \sum_{i=1}^n K_i(\zeta_1,\ldots,\zeta_i). \nonumber \end{eqnarray}
We know from Lemma \ref{termslem} that 
$$ u_n = U_n + w,$$
where $w$ is a sum of multilinear operators, all of order $> n$.  To find $U_{n+1}$, we use the iteration 
$$ u_{n+1} = u_0 + A(U_n+w)+ B (U_n + w) \otimes (U_n + w) \otimes (U_n+ w) .$$
We know (also from Lemma \ref{termslem} ) that $K_{n+1}$ will be the sum of all terms here which are of order $n+1$. Since $w$ contains only terms of order $\geq n+1$, after applying $A$ or $B$, the result will be of higher order and hence will not be included in $K_{n+1}$.  So any term containing $w$ after expanding out the tensor product can be dropped, and we have that all terms of $K_{n+1}$ will be contained in the sum $$AU_n + B U_n\otimes U_n \otimes U_n.$$ Since $A$ and $B$ each add one to the order, $K_{n+1}$ will consist of $AK_n$ and all terms of the form $$B K_{i_1}\otimes K_{i_2} \otimes K_{i_3} $$
where the ordered triplets $(i_1, i_2, i_3)$ are such that $i_1+i_2+i_3 = n$.  Hence we have derived the following:
\begin{eqnarray} K_0 &=& u_0 ,\nonumber \\   K_1 &=& u_0 + Au_0 + B u_0\otimes u_0\otimes u_0 \nonumber, \\ 
K_{n+1} &=&  AK_n + B\sum_{ \substack{ (i_1, i_2, i_3) \\ i_1+i_2+i_3 =n \\ 0\leq i_1,i_2,i_3 \leq n}} K_{i_1}\otimes K_{i_2}\otimes K_{i_3}. \label{Kformula} \end{eqnarray}
We note that we can count the number of such ordered triples in the above sum to be $$C(n):= n(n+1)/2 + (n+1). $$

 \subsection{Bounds on the forward operators.}
In order to analyze the inverse Born series, we will need bounds for the forward operators $K_i$. 
We will see that to apply existing convergence results about the inverse Born series we need boundedness of the operators as multilinear forms. We use the notation $| \cdot  |_\infty $ to denote the bound on any multilinear operator of order $n$ as follows : 
  \begin{definition} For any multilinear operator $K$ of order $n$ on $[L^\infty(\Omega)]^{2n}$, we define 
  $$ | K |_\infty = \sup_{\substack{ \zeta_1,\ldots,\zeta_n }}  { \| K(\zeta_1,\dots\zeta_n) \| \over{ \| \zeta_1\|\cdots\| \zeta_n\|}} . $$
  \end{definition} Note that, for two multilinear operators $T_1$ and $T_2$ of the same order, we have the triangle inequality
  $$| T_1+T_2|_\infty \leq | T_1 |_\infty + |T_2|_\infty.$$

 \begin{claim} The forward operator $K_n$ given by (\ref{Kformula}) is a bounded multilinear operator from $[L^\infty(\Omega)]^{2n}$ 
 to $C(\partial\Omega\times\partial\Omega)$ and
 \begin{equation}\label{Knbound} | K_n |_\infty  \leq \nu_n \mu^{n} \end{equation}
 where \begin{equation}\label{mudef}  \mu = k^2\sup_{x\in \Omega} \int_\Omega | G(x,y) | dy, \end{equation},
 $$ \nu_0 = \| u_0 \|_{C(\overline{\Omega}\times\partial\Omega )}, $$
 and for all $n\geq 1$,
 \begin{equation} \label{nudef} \nu_{n+1} = \nu_n + \sum_{ \substack{  (i_1,i_2,i_3) \\ i_1+i_2+i_3 =n \\ 0\leq i_1,i_2,i_3 \leq n}} \nu_{i_1}\nu_{i_2}\nu_{i_3} . \end{equation}
 \end{claim}
 \begin{proof} We first note that for our product operators in Definitions \ref{def1} and \ref{def2}, we have that $$ | BT_l |_\infty \leq \mu | T_l |_\infty , $$  $$ | AT_l |_\infty \leq \mu | T_l |_\infty , $$ and
 $$ | T_l \otimes T _j |_\infty \leq | T_l |_\infty | T_j |_\infty. $$
 The proof works then, by induction. The base case clearly holds with trivially with  $| K_0 |_\infty = \nu_0 $.  Assume that for each $i\leq n$, we have
 $$ \| K_i \| \leq \nu_i \mu^i.$$ Using (\ref{Kformula}), we obtain 
\begin{eqnarray}   | K_{n+1} |_\infty &\leq&  |  AK_n |_\infty + | B \sum_{ \substack{ (i_1,i_2,i_3) \\ i_1+i_2+i_3 =n \\ 0\leq i_1,i_2,i_3 \leq n}}  K_{i_1}\otimes K_{i_2}\otimes K_{i_3} |_\infty \nonumber \\ &\leq& \mu | K_n |_\infty + \mu \sum_{ \substack{ (i_1,i_2,i_3)\\ i_1+i_2+i_3 =n \\ 0\leq i_1,i_2,i_3 \leq n}}  | K_{i_1} |_\infty | K_{i_2}|_\infty | K_{i_3} |_\infty \nonumber \end{eqnarray}
which gives, by the inductive hypothesis 
\begin{eqnarray}   | K_{n+1} |_\infty &\leq&  \nu_n \mu^{n+1} + \mu^{n+1} \sum_{ \substack{ (i_1,i_2,i_3) \\ i_1+i_2+i_3 =n \\ 0\leq i_1,i_2,i_3 \leq n}}  \nu_{i_1}\nu_{i_2}\nu_{i_3} \nonumber \\ 
&=& \nu_{n+1}\mu^{n+1}. \nonumber \end{eqnarray} 
 \end{proof}
 \begin{claim} \label{nubound} For the sequence $\{ \nu_n\} $ given by (\ref{nudef}) There exist constants $K$ and $\nu$ (both depending on $\nu_0$ but independent of $n$) such that for any $n\geq 0$, $$ {\nu_n} \leq \nu K^n .$$
 \end{claim} 
\begin{proof} To prove this, we consider the generating function 
 $$ P(x) = \sum_{n=0}^\infty \nu_n x^n .$$ 
 We first note that it suffices to prove that this power series has a positive radius of convergence, since if this is the case, then for some positive $x$ the terms $\nu_n x^n 
 \rightarrow 0 $. In particular, they are bounded by some $\nu$, which would imply that $$\nu_n \leq \nu (1/x)^n.$$ 
 We now show that $P(x)$ is analytic in some nontrivial interval around zero. Consider, formally for now, the cube of $P$,
\begin{eqnarray} (P(x))^3 &=& \sum_{ \substack{  i_1,i_2,i_3 = 0,\ldots , \infty } } x^{i_1} x^{i_2} x^{i_3}  \nu_{i_1}\nu_{i_2}\nu_{i_3} \nonumber \\
&=& \sum_{n=0}^\infty f_n x^n \nonumber \end{eqnarray}
where $$ f_n = \sum_{ \substack{ (i_1,i_2,i_3) \\  i_1+i_2+i_3 =n \\ 0\leq i_1,i_2,i_3 \leq n}} \nu_{i_1}\nu_{i_2}\nu_{i_3} ,$$
which is exactly as appears in (\ref{nudef}).
Aow, we multiply  (\ref{nudef}) by $x^n$ and sum to obtain
 $$ \sum_{n=0}^\infty \nu_{n+1} x^n =   \sum_{n=0}^\infty \nu_n x^n + \sum_{n=0}^\infty f_n x^n.$$
 One checks that the left hand side is simply $(P(x) - \nu_0)/x$, and so the above yields
$$ (P(x) - \nu_0)/x = P(x) + (P(x))^3 . $$
So we have
  \begin{equation}\label{polynomial}  x(P(x))^3 + (x-1) P(x) + \nu_0 =0. \end{equation}
  This polynomial in $P$ is singular, so it is not clear that it has an analytic solution at $x=0$. However, if we differentiate with respect to $x$, we obtain
  \begin{equation} \label{ode} P^\prime (x) = - {(P(x))^3+P(x) \over{3x(P(x))^2 + x-1 }} \end{equation}
  with $P(0) = \nu_0$. Since the right hand side is an analytic function of  $x$ and $P$ in a neighborhood of $(0,\nu_0)$, the ode (\ref{ode}) (with initial condition)  has a unique analytic solution in a neighborhood of  $x=0$ (see for example Theorem 4.1 of Teschl \cite{Te}) . Integration of (\ref{ode}) combined with the initial condition implies that this analytic solution satisfies (\ref{polynomial}), and hence its coefficients must satisfy (\ref{nudef}). 
  \end{proof}

\begin{proposition} \label{forwardopbounds} The forward operator $K_n$ given by (\ref{Kformula}) is a bounded multilinear operator from $[L^\infty(\Omega)]^{2n}$ to $C^0(\overline{\Omega}\times\partial\Omega)$, and its bound satisfies 
 \begin{equation}\label{Knbound} | K_n |_\infty  \leq \nu( K \mu)^n ,\end{equation}
 where \begin{equation}\label{mudef}  \mu = k^2\sup_{x\in \Omega} \int_\Omega | G(x,y) | dy \end{equation} and $\nu,K,$ both depending on 
 $ \nu_0 = \| u_0 \|_{C(\overline{\Omega}\times\partial\Omega )}$, are as in Lemma \ref{nubound}.
\end{proposition}
  \begin{corollary} The forward Born series
  $$ u = u_0 + \sum_{n=1}^\infty K_n (\zeta,\ldots,\zeta) $$
  where $K_n$ are given by (\ref{Kformula}) converges in  $C(\overline{\Omega})$ for 
 $$ \| \zeta\|_\infty \leq  {1\over{K\mu}} $$
 where \begin{equation}\label{mudef}  \mu = k^2\sup_{x\in \Omega} \int_\Omega | G(x,y) | dy, \end{equation} and $\nu,K,$ both depending on 
 $ \nu_0 = \| u_0 \|_{C(\overline{\Omega})}$, are as in Lemma \ref{nubound}.
\end{corollary}

\section{Inverse Born Series}
The inverse problem is to reconstruct the coefficients $\alpha$ and $\beta$ from measurements of the scattering data $\phi$. We proceed by recalling that the Inverse Born series is defined as
\begin{equation}\label{inversedefinition}
    \tilde{\zeta} = \mathcal{K}_1 \phi + \mathcal{K}_2 (\phi)  + \mathcal{K}_3( \phi)  + \cdots \ ,
\end{equation}
where the data $\phi  \in C(\partial\Omega\times\partial\Omega).$ The inverse series was analyzed in \cite{moskow_1} and later studied in \cite{HoSc}. The inverse operator $\mathcal{K}_m$  can be computed from the formulas 
\begin{align}
\label{inv_operators}
\mathcal{K}_1 (\phi) &= K_1^{+} (\phi),\\
\mathcal{K}_2(\phi) &=-\mathcal{K}_1\left(K_2 (\mathcal{K}_1(\phi),\mathcal{K}_1(\phi))\right),\\
\mathcal{K}_m(\phi) &= -\sum_{n=2}^{m}\sum_{i_1+\cdots+i_n = m}  \K_1{K}_n \left( \mathcal{K}_{i_1}(\phi), \dots, \mathcal{K}_{i_n}(\phi) \right) .
\end{align}
Here $K_1^+$ denotes a regularized pseudoinverse of $K_1$, and $\tilde{\zeta}$ is the series sum, an approximation to $\zeta$, when it exists. 
Recall that this inverse series requires forward solves for the background problem only (i.e., applying the forward series operators), and requires a pseudoinverse and regularization of the first linear operator only; $ \mathcal{K}_1 = K_1^+$. 
The bounds on the forward operators from Proposition \ref{forwardopbounds} combined with Theorem 2.2 of \cite{HoSc}  yield the following result.

\begin{theorem}
If   $\Vert \mathcal{K}_1 \phi\Vert < r $, where 
where the radius of convergence $r$ is given by
$$
r=\frac{1}{2K\mu} \left[\sqrt{16 C^2+1}-4 C \right],
$$
where $C = \max\{2,\|\mathcal{K}_1\|\nu K\mu\}.$
\begin{equation}  \mu = k^2\sup_{x\in \Omega} \int_\Omega | G(x,y) | dy, \end{equation} and $\nu,K,$ both depending on 
 $ \nu_0 = \| u_0 \|_{C^0(\Omega\times\partial\Omega )}$, are as in Lemma \ref{nubound}, then the inverse Born series converges. 
\end{theorem}

\section{Numerical Experiments}
In this section we present several numerical experiments using the inverse Born series (\ref{inversedefinition}) to reconstruct $\alpha$ and $\beta$ from synthetic data. In all cases, used the FEniCS PDE solver library in Python to create the synthetic data $\phi$. To implement the inverse series we also need the forward operators; for these we use the recursive formula, implemented as in Algorithm \ref{forwardalg}. 
\begin{algorithm}

    \Fn{compute-K$(n, \alpha, \beta)$}{
        \If{$n = 0$}{
            \Return $u_0$\;
        }
        
        $v_\alpha$ :=  find-K$(n-1, \alpha(n), \beta(n))$\;
        
        $v_\beta$ := 0\;
        \For{$i_1$ = 0 to $n-1$}{
            \For{$i_2$ = 0 to $n-i_1-1$}{
                $i_3 := n - i_2 - i_1 - 1$\;
                $K_{i_1} := $compute-K$(i_1, \alpha(1:i_1), \beta(1:i_1)$\;
                $K_{i_2} := $compute-K$(i_2, \alpha(i_1+1:i_1+i_2), \beta(i_1+1:i_1+i_2))$\;
                $K_{i_3} := $compute-K$(i_3, \alpha(i_1+i_2+1:n-1), \beta(i_1+i_2+1:n-1))$\;
                $v_\beta := v_\beta + K_{i_1} \cdot K_{i_2} \cdot K_{i_3}$\;
            }
        }
        \Return $A(v_\alpha, \alpha(n)) + B(v_\beta, \beta(n))$\;
        
    }
    \vskip .3in
        \caption{Generation of the terms in the forward series.}   
       \label{forwardalg}  \end{algorithm}
        
We implement the application of the operators $A$ and $B$ by solving a background PDE (equivalent to integrating against background Green's function kernel), again using the FEniCS PDE solver library in Python; taking care to choose a different FEM mesh from those used for the generation of the synthetic data. The inverse Born series implementation is the same as in previous work, see for example \cite{}, here calling on the above Algorithm \ref{forwardalg} to call the forward operators.

We begin with an experiment in one dimension, on the interval $\Omega = [0, 1]$.
Here, we have only two points on the boundary, meaning we can only take samples at these two points. However, thanks to the nonlinearity, a scaled source has the potential to yield more information.  For example, if we scale the source and take the right linear combination of the two solutions, we eliminate $\alpha$. While we don't implement this explicitly, we do capitalize on scaling for the one dimensional example, where we found it improved the reconstruction.

In the following example, we used $12$ scaled sources on each side of the interval and three different frequencies $k=0.9,1,1.1$, for a total of $72$ sources.
We chose the reference functions $\alpha$ and $\beta$ to be
\[
    \alpha, \beta = \begin{cases}
        \frac{\gamma}{\sqrt{2\pi \epsilon}}e^{\frac{-x^2}{2\epsilon }} & \text{if } 0.4 \leq x \leq 0.6\\
        0&\text{otherwise}
    \end{cases}
\]
for $\epsilon=0.01$ and $\gamma=0.2$. We see their simultaneous reconstruction in Figure \ref{fig:1d-sources}.
\begin{figure}[H] \label{fig:1d-sources}
\centering
\includegraphics[height=2in]{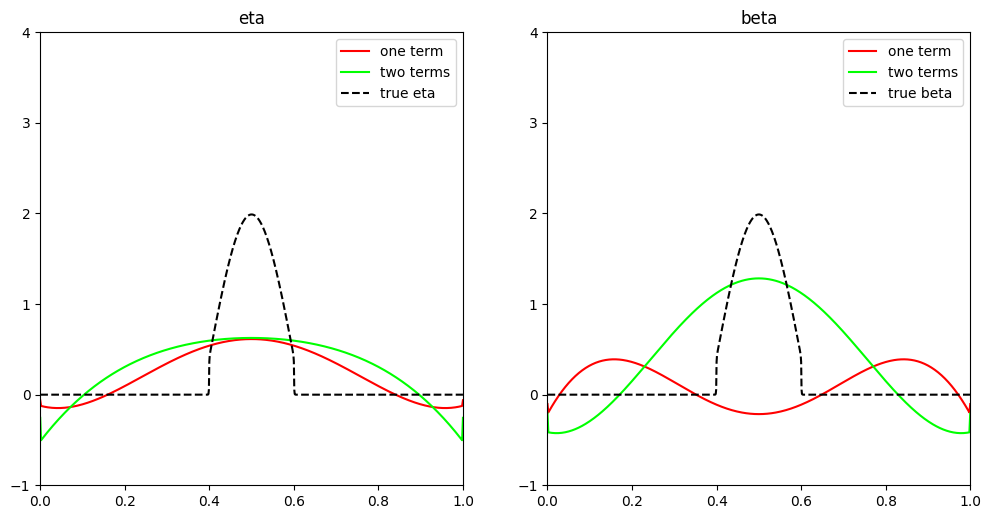}
\caption{One dimensional simultaneous reconstruction of $\alpha$ and $\beta$ using varying source strength and frequency}
\label{fig:1d-sources}
\end{figure}
Next we run several experiments in two dimensions, all with domain  $\Omega$ the unit disk.
First, we compare the reconstruction of $\alpha$ with $\beta = 0$  (traditional linear problem), to the reconstruction of $\beta$ with $\alpha=0$.  In both cases we let the unknown be a piecewise constant. For the first example, we chose the moderate contrast medium (\ref{refmedium1}),  and we see reconstruction results in Figure \ref{fig:low}. Next, we increase the contrast by a factor of $4$ in (\ref{refmedium1}) and we see results in Figures \ref{fig:medium} and \ref{fig:medium-cross}. Finally, we choose a very high contrast,  the medium (\ref{refmedium1}) multiplied by a factor of $16$. In Figures \ref{fig:high} and \ref{fig:high-cross}, we see that the method does not produce good results for $\alpha$, but is not as bad for $\beta$.

\begin{equation} \label{refmedium1}
    \beta,\alpha = \begin{cases}
        1 & \text{if } (x-0.3)^2 + y^2 \leq 0.2\\
        0 & \text{otherwise}
    \end{cases}
\end{equation}

\begin{figure}[H]
    \centering
    \begin{subfigure}{1\textwidth}
        \centering
        \includegraphics[height=2in]{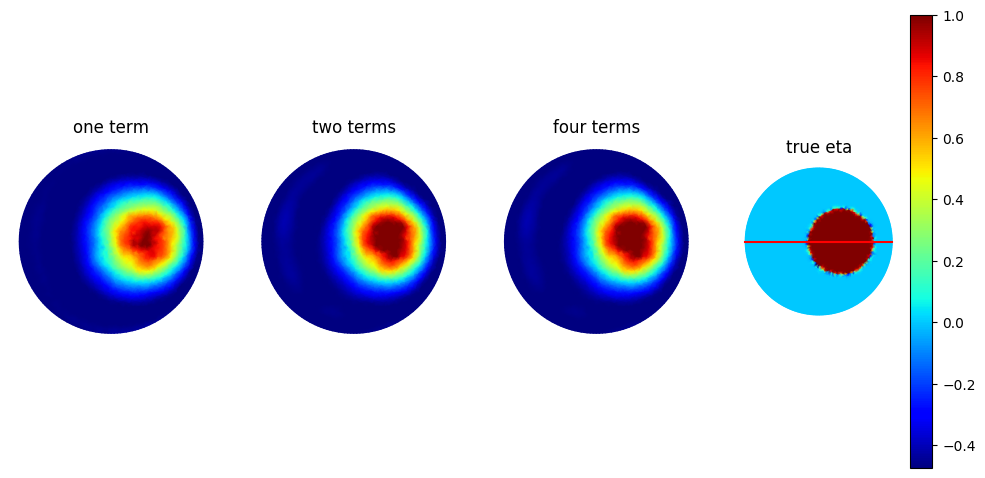}
        \caption{$\alpha$ reconstruction with $\beta = 0$}
        \label{fig:sub1}
    \end{subfigure}
    \begin{subfigure}{1\textwidth}
        \centering
        \includegraphics[height=2in]{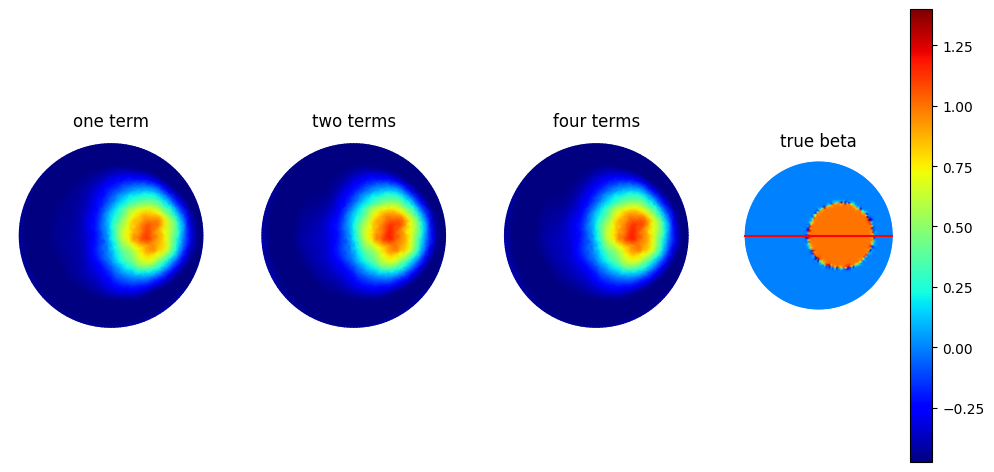}
        \caption{$\beta$ reconstruction with $\alpha = 0$}
        \label{fig:sub2}
    \end{subfigure}
    \caption{Independent reconstruction only for low contrast}
    \label{fig:low}
\end{figure}

\begin{figure}[H]
    \centering
    \begin{subfigure}{.5\textwidth}
        \centering
        \includegraphics[height=1.5in]{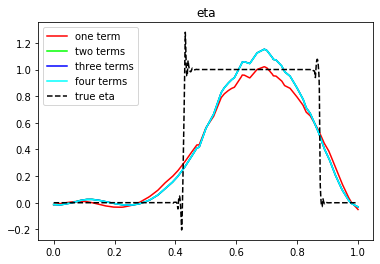}
        \caption{$\alpha$ reconstruction with $\beta = 0$}
        \label{fig:sub1}
    \end{subfigure}%
    \begin{subfigure}{.5\textwidth}
        \centering
        \includegraphics[height=1.5in]{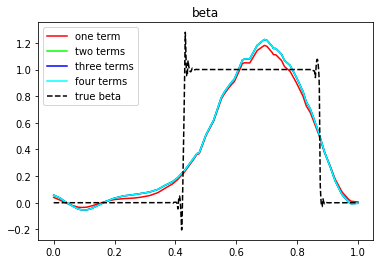}
        \caption{$\beta$ reconstruction with $\alpha = 0$}
        \label{fig:sub2}
    \end{subfigure}
    
    \caption{Cross-section of independent reconstruction for low contrast}
    \label{fig:low-cross}
\end{figure}
\begin{figure}[H]
    \centering
    \begin{subfigure}{1\textwidth}
        \centering
        \includegraphics[height=2in]{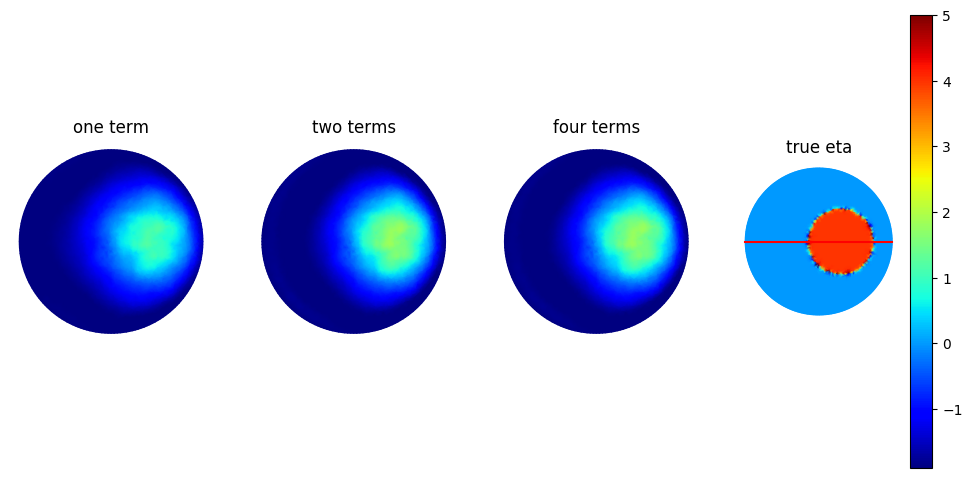}
        \caption{$\alpha$ reconstruction with $\beta = 0$}
        \label{fig:sub1}
    \end{subfigure}
    \begin{subfigure}{1\textwidth}
        \centering
        \includegraphics[height=2in]{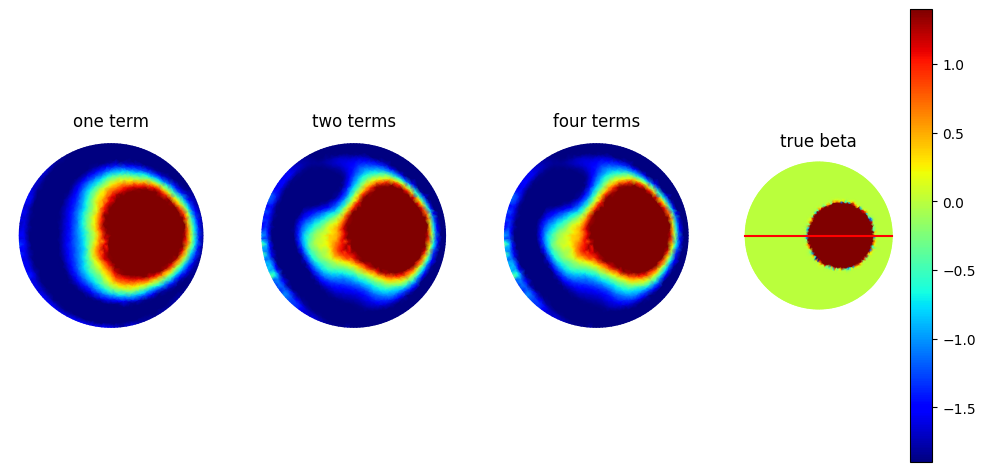}
        \caption{$\beta$ reconstruction with $\alpha = 0$}
        \label{fig:sub2}
    \end{subfigure}
    \caption{Independent reconstruction for medium contrast}
    \label{fig:medium}
\end{figure}

\begin{figure}[H]
    \centering
    \begin{subfigure}{.5\textwidth}
        \centering
        \includegraphics[height=1.5in]{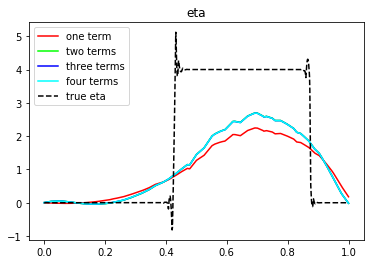}
        \caption{$\alpha$ reconstruction with $\beta = 0$}
        \label{fig:sub1}
    \end{subfigure}%
    \begin{subfigure}{.5\textwidth}
        \centering
        \includegraphics[height=1.5in]{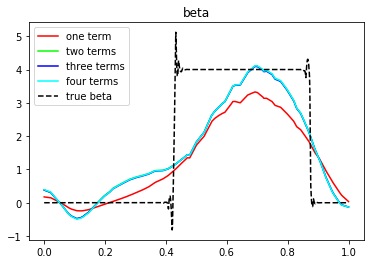}
        \caption{$\beta$ reconstruction with $\alpha = 0$}
        \label{fig:sub2}
    \end{subfigure}
    \caption{Cross-section of independent reconstruction only for medium contrast}
    \label{fig:medium-cross}
\end{figure}
\begin{figure}[H]
    \centering
    \begin{subfigure}{1\textwidth}
        \centering
        \includegraphics[height=2in]{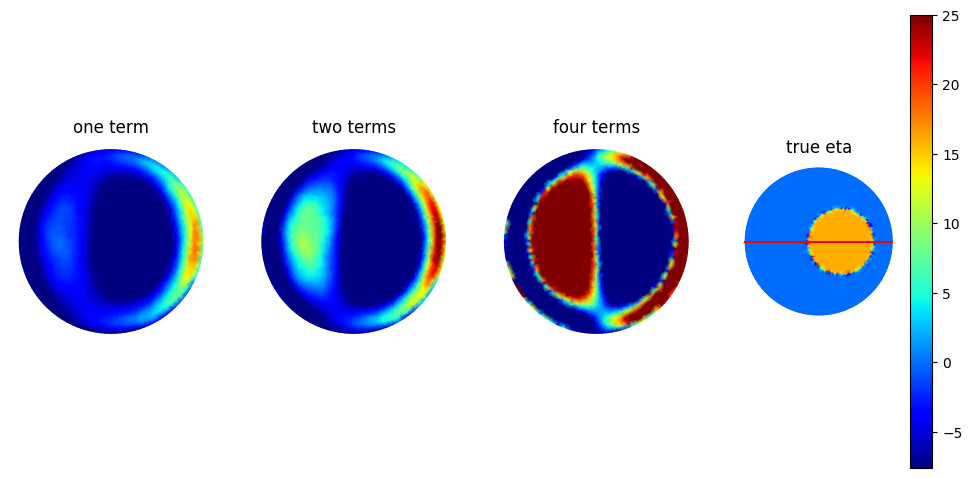}
        \caption{$\alpha$ reconstruction with $\beta = 0$}
        \label{fig:sub1}
    \end{subfigure}
    \begin{subfigure}{\textwidth}
        \centering
        \includegraphics[height=2in]{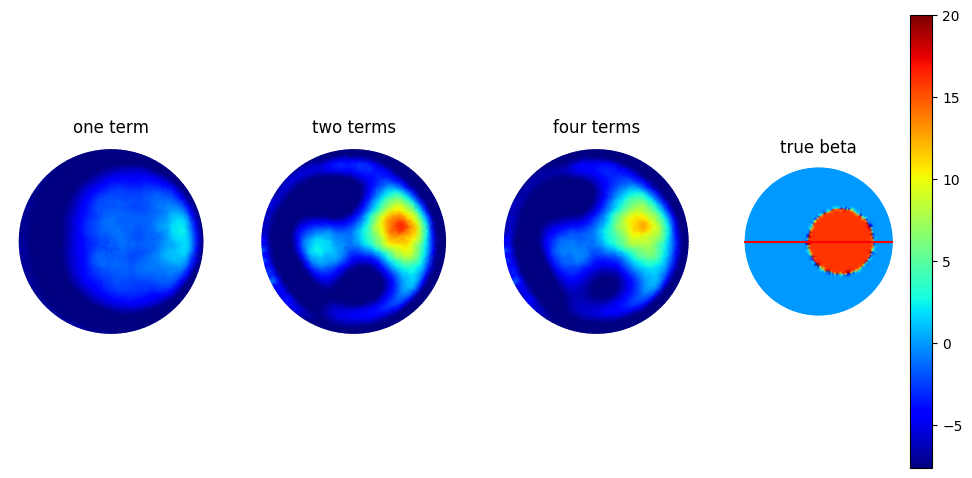}
        \caption{$\beta$ reconstruction with $\alpha = 0$}
        \label{fig:sub2}
    \end{subfigure}
    \caption{Independent reconstruction for high contrast}
    \label{fig:high}
\end{figure}

\begin{figure}[H]
    \centering
    \begin{subfigure}{.5\textwidth}
        \centering
        \includegraphics[height=1.5in]{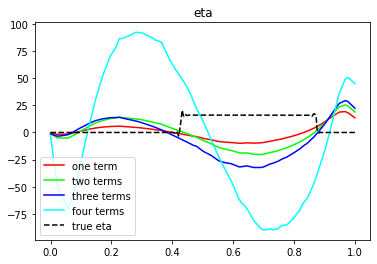}
        \caption{$\alpha$ reconstruction with $\beta = 0$}
        \label{fig:sub1}
    \end{subfigure}%
    \begin{subfigure}{.5\textwidth}
        \centering
        \includegraphics[height=1.5in]{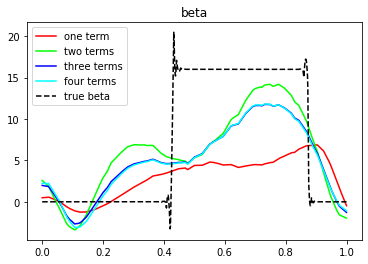}
        \caption{$\beta$ reconstruction with $\alpha = 0$}
        \label{fig:sub2}
    \end{subfigure}
    \caption{Cross-section of independent reconstruction for high contrast}
    \label{fig:high-cross}
\end{figure}

In this next set of examples, we consider the simultaneous reconstruction of $\beta$ and $\alpha$.  In the first example, we take
\begin{equation}\label{medium2}
    \beta = \alpha =  \frac{2}{\sqrt{2\pi\epsilon}}\exp{\left(\frac{-|x-x_0|^2}{2\epsilon}\right)}
\end{equation}
for $x_0=( -.3,3)$ and $\epsilon=0.04$. We see results in Figures \ref{fig:both-medium} and \ref{fig:both-medium-cross}. In the second example, we raise the contrast in (\ref{medium2}) by a factor of $4$ ,and we see that we still get reasonable reconstructions in Figure \ref{fig:both-high} and Figure \ref{fig:both-high-cross}.
\begin{figure}[H]
    \centering
    \includegraphics[height=2in]{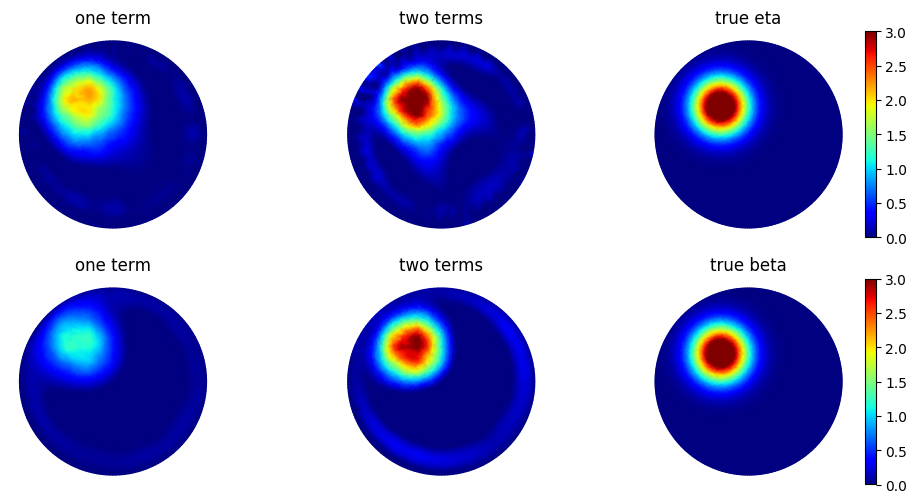}
    \caption{Simultaneous $\alpha$ and $\beta$ reconstruction for high contrast}
    \label{fig:both-medium}
\end{figure}

\begin{figure}[H]
    \centering
    \includegraphics[height=2in]{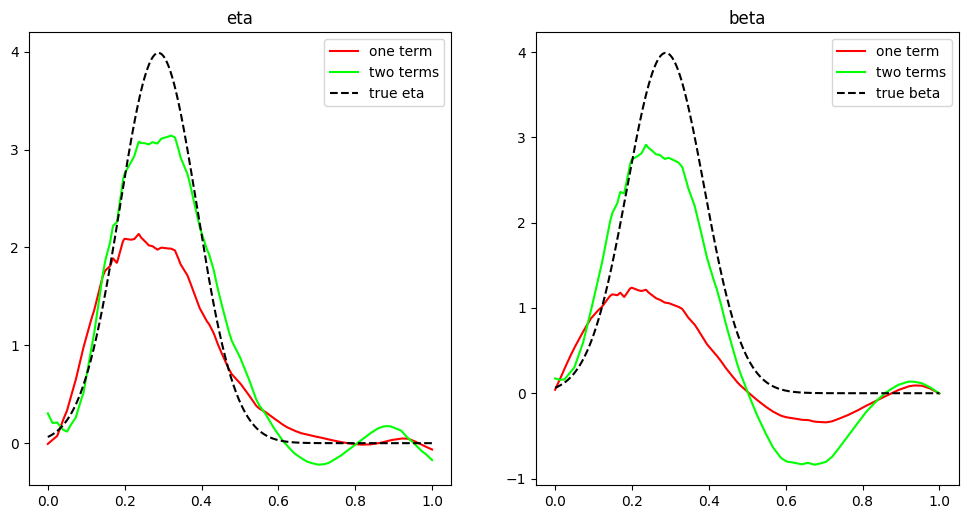}
    \caption{Cross-section $\alpha$ and $\beta$ reconstruction for high contrast}
    \label{fig:both-medium-cross}
\end{figure}
\begin{figure}[H]
    \centering
    \includegraphics[height=2in]{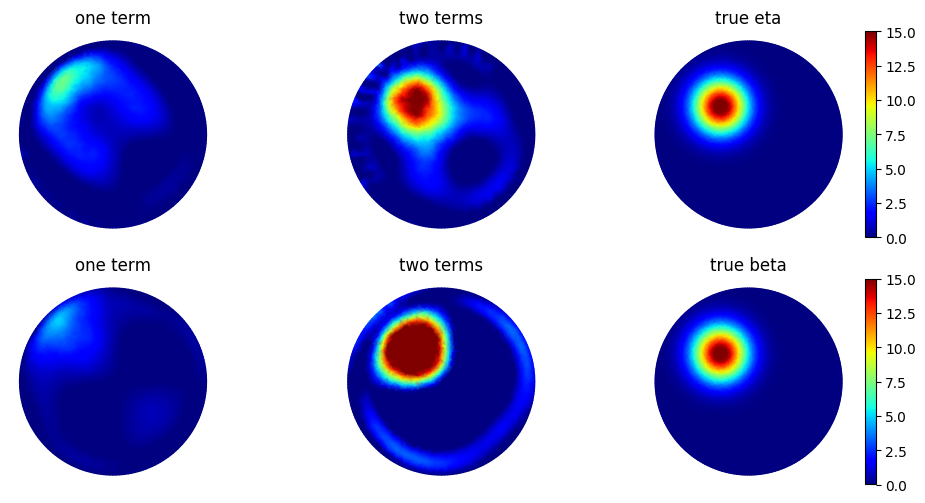}
    \caption{Simultaneous $\alpha$ and $\beta$ reconstruction for very high contrast}
    \label{fig:both-high}
\end{figure}

\begin{figure}[H]
    \centering
    \includegraphics[height=2in]{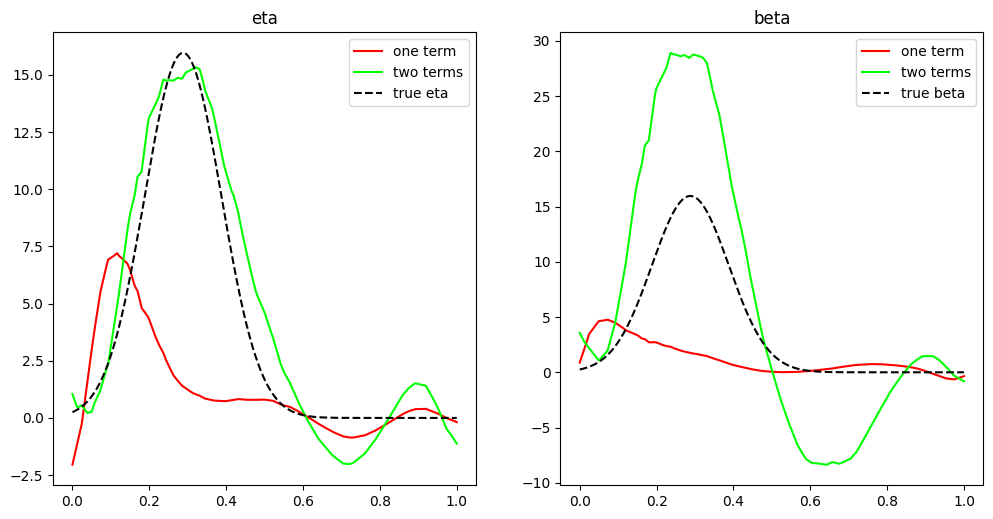}
    \caption{Cross-section $\alpha$ and $\beta$ reconstruction for very high contrast}
    \label{fig:both-high-cross}
\end{figure}

\section{Discussion}
 We have considered the Born and inverse Born series for scalar waves with a cubic nonlinearity of Kerr type. We found a recursive formula for the forward operators in the Born series. This result gives conditions which guarantee convergence of the Born series, and also leads to conditions which guarantee convergence of the inverse Born series. Our results are illustrated results with numerical experiments.

The ideas developed here provide a framework for studying inverse problems for a wide class of nonlinear PDEs with polynomial nonlinearities. The formulas and algorithm for generating the forward operators, the use of the generating functions, and the resulting reconstruction algorithm are readily generalizable to this setting and will be explored in future work.

\section{Acknowledgments}  
The authors are indebted to Jonah Blasiak and R. Andrew Hicks for discussions essential to the proof of Lemma \ref{nubound}. S. Moskow was partially supported by the NSF grant DMS-2008441. J. Schotland was supported in part by the NSF grant DMS-1912821 and the AFOSR grant FA9550-19-1-0320.

\appendix
\section{Proof of Proposition~1} \label{fixedpointappendix}
In this appendix we obtain conditions for existence of a unique solution to (\ref{integralequation}) and give alternative conditions  on $\alpha$ and $\beta$ that guarantee convergence of the Born series.  
Define the linear operator $$ G: C^0 (\overline{\Omega})\rightarrow C^0(\overline{\Omega}) $$ by $$ G(v) = -k^2\int_\Omega G(x, y) v(y)\,dy. $$ Then, for $u_0$ in $C^0(\overline{\Omega})$, we have that $T$ can be written as 
\begin{equation}
	T(v) = u_0 +G\left(\alpha v + \beta v^3\right).
\end{equation}
Note that $G$ is compact and bounded.  Define $\mu$ by 
\begin{equation}\label{mudef1}  \mu = k^2\sup_{x\in \Omega} \int_\Omega | G(x,y) | dy, \end{equation}
Then we have that $$ \| G(v) \| \leq \mu\Vert v \Vert$$ for all $v\in C^0(\overline{\Omega})$. 	
We will make use of the following two lemmas. The first gives conditions to have a contraction. 
\begin{claim}\label{lipbound}
For any $f,g\in C^0(\overline{\Omega})$ such that $\| f\|,\| g\| \leq R$, we have 
$$ \| T(f)-T(g) \| \leq q \| f-g\|$$
where $$q = \mu (\left\Vert\alpha\right\Vert + 3R^2\left\Vert\beta\right\Vert) $$ 
for  $\mu$  defined by (\ref{mudef1}).
\end{claim}
\begin{proof} 
Let $f, g \in B$. Then,  
		 \begin{equation}
		 	T(f)- T(g)) = 
		  k^2\int_{\Omega}G(x, y)\left[\alpha(y)\left(f(y) - g(y)\right) + \beta(y)\left(f(y)^3 - g(y)^3\right)\right]\,dy \end{equation}
		  so that 
		  \begin{multline} 
		  \| T(f)-T(g)\| 
		 	\leq \mu\left\Vert \alpha\left(f - g\right) + \beta\left(f^3 - g^3\right)\right\Vert\\
		 	= \mu\left\Vert\left[\alpha + \beta\left(f^2 + fg + g^2\right)\right](f - g)\right\Vert\\
		 	\leq \mu\left(\left\Vert\alpha\right\Vert + \left\Vert\beta\left(f^2 + fg + g^2\right)\right\Vert\right)\cdot\left\Vert(f - g)\right\Vert\\	 		 	\leq \mu\left(\left\Vert\alpha\right\Vert + 3R^2\left\Vert\beta\right\Vert\right)\cdot\Vert f - g \Vert. 
		 \end{multline}
\end{proof}
The second lemma gives us a ball which $T$ maps into itself. 
\begin{claim}\label{mapping}
Let $r>0$ be given, and let $B=B(u_0, r)$ be the ball of radius $r$ about $u_0$ in $C^0(\Omega)$.  Define $R=\Vert u_0 \Vert + r$. Then if  $$\mu R\left(\Vert \alpha \Vert + R^2\Vert\beta\Vert\right) < r,$$ we have that $T(v) \in B$ for any $v\in B$, and hence $T$ has a fixed point in $B$. 
\end{claim}
	\begin{proof}
Assume that $v \in B$. Then, by the triangle inequality, we have that $\Vert v \Vert \leq R$. Then, 
		\begin{eqnarray}
			\| T(v)-u_0\|  &=&  \| G(\alpha v+\beta v^3) \|  \\
			&\leq& \mu\left\Vert\alpha v + \beta v^3\right\Vert\\
			&\leq& \mu\left(R\Vert\alpha\Vert + R^3\Vert \beta \Vert\right).
		\end{eqnarray}
By hypothesis, this is less than $r$, and hence $T(v) \in B$.
	\end{proof}
\begin{claim}\label{contraction}
		If the hypotheses of Lemma \ref{mapping} hold and, additionally,  $$\mu\left(\Vert\alpha\Vert + 3R^2\Vert\beta\Vert\right) < 1,$$ then $T$ is a contraction mapping on $B$.
	\end{claim}
	\begin{proof}
Since $\mu\left\Vert\alpha\right\Vert + 3R^2\left\Vert\beta\right\Vert < 1$ by assumption, by Lemma \ref{lipbound} we have that $d(T(f), T(g)) \leq q\cdot d(f, g)$ where $q<1$.  
Since $T$ also maps $B$ into itself, $T$ is a contraction mapping on $B$. \end{proof}
Clearly if Lemma \ref{contraction} holds,  by the Banach fixed point theorem $T$ will have a unique fixed point on $B$.  Furthermore, if we start with initial function $u_0$ in $B$, then fixed-point iteration will converge to the unique fixed point, which in this case is the solution to the integral  of the partial differential equation defined by  (\ref{integralequation}).  The iterates of the fixed point iteration will generate the (forward) Born series. 

\begin{remark}  We note that If $\beta$ and $\alpha$ satisfy the (more restrictive) condition $$\mu(\Vert \alpha\Vert + 3R^2\Vert\beta\Vert) < \frac{r}{R},$$ then both Lemma \ref{mapping} and Lemma \ref{contraction} are also satisfied.
\end{remark}
The following is well known for the linear case, see for example \cite{CoKr}.
\begin{proposition}If $\beta=0$ and $\Vert\alpha\Vert < \frac{1}{\mu}$, then $T$ has a unique fixed point on all of $C^0(\Omega)$.	\end{proposition}	
	\begin{proof}

In this case, Lemma \ref{mapping} and Lemma \ref{contraction} are both satisfied if 
$$ \mu \Vert \alpha\Vert < \frac{r}{r+\| u_0 \|}  .$$  Since  by assumption $ \mu  \Vert\alpha\Vert < 1$,  there exists $r_0$ such that for any $r > r_0$ 
$$ \mu  \Vert\alpha\Vert <   \frac{r}{r+\| u_0 \|}  < 1. $$
Therefore $T$ has a unique fixed point on $B(u_0, r)$ for any $r> r_0$, so the fixed point must be unique on all of $C^0(\Omega)$.	\end{proof}
\begin{proposition}
		If $\alpha=0$ and $\beta < \frac{4}{27\mu\Vert u_0\Vert^2}$, then $T$ has a unique fixed point in the ball $B(u_0, \Vert u_0 \Vert/2)$. 
	\end{proposition}
	\begin{proof}
		Let $r=\Vert u_0\Vert/2$. This means that $R=3\Vert u_0\Vert/2$. Then, 
		\begin{align*}
			\mu R\left(\Vert\alpha\Vert+R^2\Vert\beta\Vert\right) &= \mu R^3\Vert\beta\Vert\\
			&=\mu (3\Vert u_0\Vert/2)^3\Vert\beta\Vert\\
			&< \frac{27\mu \Vert u_0\Vert^3}{8}\cdot\frac{4}{27\mu \Vert u_0\Vert^2}\\
			&= \frac{\Vert u_0\Vert}{2} = r
		\end{align*}
Thus, the hypothesis needed for Lemma \ref{mapping} is satisfied. Additionally, we have that 
		\begin{align*}
			\mu \left(\Vert\alpha\Vert + 3R^2\Vert\beta\Vert\right) &= 3R^2\mu \Vert\beta\Vert\\
			&= 3(3\Vert u_0\Vert/2)^2\mu \Vert\beta\Vert\\
			&< \frac{27\mu \Vert u_0\Vert^2}{4}\cdot \frac{4}{27\mu \Vert u_0\Vert^2}\\
			&= 1,
		\end{align*}
so that the condition for Lemma \ref{contraction} to hold is also satisfied, and hence $T$ has a unique fixed point on $B(u_0, r)$. \end{proof}
\begin{proposition}\label{boundsapp}
If there exists some $\gamma > 1/2 $ such that $$\Vert\alpha\Vert < \frac{2\gamma-1}{2\mu (1+\gamma)}$$ and $$\Vert\beta\Vert < \frac{1}{2\mu \Vert u_0\Vert^2(1+\gamma)^3},$$ then $T$ has a unique fixed point in the ball $B(u_0, \gamma\Vert u_0\Vert)$.\end{proposition}
\begin{proof}
Let $r=\gamma\Vert u_0\Vert$, which that $R=(1+\gamma)\Vert u_0\Vert$.
The hypotheses then imply that Lemma \ref{mapping} holds, because 
		\begin{align*}
			\mu R(\Vert\alpha\Vert + R^2\Vert\beta\Vert) &=\mu (1+\gamma)\Vert u_0\Vert(\Vert\alpha\Vert + (1+\gamma)^2\Vert u_0\Vert^2\Vert\beta\Vert)\\
			&= \mu (1+\gamma)\Vert u_0\Vert\left(\frac{2\gamma-1}{2\mu (1+\gamma)}+\frac{1}{2\mu (1+\gamma)}\right)\\
			&= \Vert u_0\Vert\left(\frac{2\gamma-1}{2}+\frac{1}{2}\right)\\
			&= \gamma\Vert u_0\Vert = r.		
		\end{align*}
For Lemma \ref{contraction}, we have that
		\begin{align*}
			\mu (\Vert\alpha\Vert + 3R^2\Vert\beta\Vert) &= \mu (\Vert\alpha\Vert+3(1+\gamma)^2\Vert u_0\Vert^2\Vert\beta\Vert)\\
			&< \mu \left(\frac{2\gamma-1}{2\mu (1+\gamma)} + (1+\gamma)^2\Vert u_0\Vert^2\frac{1}{2\mu \Vert u_0\Vert^2(1+\gamma)^3}\right)\\
			&= \frac{2\gamma-1}{2(1+\gamma)}+\frac{1}{2(1+\gamma)}\\
			&= \frac{\gamma}{1+\gamma} < 1.
		\end{align*}
Thus $T$ has a unique fixed point on the ball $B(u_0, \gamma\Vert u_0\Vert)$. \end{proof}


\begin{thebibliography}{99}

\bibitem{assylbekov_1}
Y.M. Assylbekov, T. Zhou, Direct and Inverse problems for the nonlinear time-harmonic Maxwell equations in Kerr-type media. Journal of Spectral Theory {\bf 11}, 1-38 (2021).

\bibitem{assylbekov_2}
Y.M. Assylbekov, T. Zhou, Inverse problems for nonlinear Maxwell’s equations with second harmonic generation, J. Diff. Eqs. {\bf 296}, 148–169 (2021).

\bibitem{boyd}
R.W. Boyd, Nonlinear Optics (Elsevier, 2008).

\bibitem{carstea}
Carstea, C., Nakamura, G., Vashisth, M., Reconstruction for the coefficients of a quasilinear elliptic partial differential equation, Appl. Math. Lett. {\bf 98}, 121–127 (2019).

\bibitem{griesmaier}
R. Griesmaier, M. Knoller and R. Mandel, Inverse medium scattering for a nonlinear Helmholtz equation. J. Math. Anal. Appl. {\bf 515}, 126356 (2022).

\bibitem{HoSc}
J. G. Hoskins and J. C. Schotland, Analysis of the Inverse Born Series: An Approach Through
Geometric Function Theory. Inverse Probl. {\bf 38}, 074001 (2022).

\bibitem{imanuvilov}
O. Imanuvilov, M. Yamamoto, Unique determination of potentials and semilinear terms of semilinear elliptic equations
from partial Cauchy data, J. Inverse Ill-Posed Probl. {\bf 21}, 85–108 (2013).

\bibitem{isakov_1}
V. Isakov, On uniqueness in inverse problems for semilinear parabolic equations, Arch. Ration. Mech. Anal. {\bf 124}
1–12 (1993).

\bibitem{isakov_2}
V. Isakov, Uniqueness of recovery of some systems of semilinear partial differential equations, Inverse Probl. {\bf 17},
607–618 (2001).

\bibitem{isakov_3}
V. Isakov and A. Nachman, Global uniqueness for a two-dimensional semilinear elliptic inverse problem, Trans. Am.
Math. Soc. {\bf 347}, 3375–3390 (1995).

\bibitem{isakov_4}
V. Isakov and J. Sylvester, Global uniqueness for a semilinear elliptic inverse problem, Commun. Pure Appl. Math. {\bf 47},
1403–1410 (1994). 

\bibitem{kang}
Kang, K. and Nakamura, G., Identification of nonlinearity in a conductivity equation via the
Dirichlet–to–Neumann map, Inverse Problems {\bf 18}, 1079–1088 (2002).

\bibitem{kurylev}
Kurylev, Y., Lassas, M., Uhlmann, G., Inverse problems for Lorentzian manifolds and non-linear hyperbolic equations, Invent. Math. {\bf 212}, 781–857 (2018).

\bibitem{lassas}
Lassas, M., Uhlmann, G., Wang, Y., Inverse problems for semilinear wave equations on Lorentzian manifolds, Comm. Math. Phys. {\bf 360}, 555–609 (2018).

\bibitem {moskow_1}
S. Moskow and J. C. Schotland, Convergence and stability of the inverse scattering series for
diffuse waves, Inverse Probl. {\bf 24}, 065005  (2008).


\bibitem{review}
S. Moskow and J. C. Schotland, Inverse Born Series, in \emph{
The Radon Transform: The First 100 Years and Beyond}, edited by R. Ramlau and O. Scherzer (De Gruyter, 2019)

 
\end{thebibliography}
\end{document}